\documentclass[11 pt]{amsart}

\usepackage{lineno,hyperref}
\usepackage{amsmath,amssymb}
\usepackage{lineno}
\usepackage{centernot}
\usepackage{graphicx}%
\usepackage{multirow}%
\usepackage{amsmath,amssymb,amsfonts}%
\usepackage{tikz-cd}
\usepackage{amsthm}%
\usepackage{mathrsfs}%
\usepackage[title]{appendix}%
\usepackage{xcolor}%
\usepackage{textcomp}%
\usepackage{manyfoot}%
\usepackage{booktabs}%
\usepackage{algorithm}%
\usepackage{algorithmicx}%
\usepackage{algpseudocode}%
\usepackage{listings}%
\usepackage{float}

\theoremstyle{plain}
\newtheorem{thm}{Theorem}[section]
\newtheorem{theorem}[thm]{Theorem}
\newtheorem{lemma}[thm]{Lemma}
\newtheorem{corollary}[thm]{Corollary}
\newtheorem{proposition}[thm]{Proposition}

\theoremstyle{definition}
\newtheorem{definition}[thm]{Definition}
\newtheorem{example}[thm]{Example}

\newtheorem{remark}[thm]{Remark}

\numberwithin{equation}{section}

\newcommand{\interior}[1]{{\kern0pt#1}^{\mathrm{o}}}

\begin{document}
	
	
	\title[Semicommutativity Via Path and Leavitt Path algebra]{Semicommutativity via Zero-Insertive Elements: Insights \\ from Path and Leavitt Path Algebras}
	
	\author[Sanjiv Subba]{Sanjiv Subba $^\dagger$}
	
	\address{$^\dagger$School of Applied Sciences\\ UPES, Dehradun\\  Bidholi-248007\\ India}
	\email{sanjivsubba59@gmail.com}
	\author[Tikaram Subedi]{Tikaram Subedi  {$^{\dagger *}$}}
	\address{$^{\dagger * }$Department of Mathematics\\  National Institute Of Technology  Meghalaya\\ Sohra, Shillong-793108\\ India}
	\email{tikaram.subedi@nitm.ac.in}
	
	.
	
	
	\begin{abstract}
		In the path algebra $KE$, where $K$ is a field and $E$ is a directed graph (or quiver), every non-loop edge can be expressed in the form $arb$, where $a,b,r\in KE$ and $ab=0$. This fact leads us to the characterization  that $KE$ is semicommutative if and only if $E$ contains no non-loop edges. For a ring $R$, let $Z_i(R)=\{x\in R: x=arb, a,b,r\in R, ab=0 \}$ and call elements of $Z_i(R)$ \textit{zero-insertive}. It follows that $R$ is semicommutative if and only if $Z_i(R)\subseteq E(R)$ and weakly semicommutative if and only if $Z_i(R)\subseteq N(R)$. We establish that every non-unit element of the Leavitt path algebra $L_K(A_2)$  is zero-insertive. If $K$ is a field and $n\geq 2$ is an integer, then each zero-insertive element of $L_K(A_n)$ is nil-clean if and only if $K\cong \mathbb{F}_2$. We call a ring $R$ \textit{zero-insertive nil clean (ZINC)} if every zero-insertive element is nil clean. We show that the Leavitt path algebra $L_{\mathbb{F}_2}(A_n)$ is a ZINC ring. Additionally, we investigate the behavior of zero-insertive elements and ZINC rings under various ring extensions. 
		
	\end{abstract}

	\maketitle
	\section{Introduction}

	
	\quad A ring $R$ is said to be \textit{semicommutative} (\cite{aas}) if $arb=0$, for all $r\in R$, whenever $ab=0$. A ring $R$ is \textit{weakly semicommutative} (\cite{wscr}) whenever $arb$ is nilpotent for all $r\in R$ whevever, $ab=0, a,b\in R$. In context of path algebras $R=KE$, where $E$ is a directed graph (or quiver) and $K$ a field, elements of the form $x=arb$, where $ab=0$, and $a,b,r\in R$ arise naturally. In particular, every non-loop edge can be represented in this form. Consequently, such elements provide a natural link between various generalizations of semicommutative rings and path algebras.
	
	For definitions and basic terminology related to path and Leavitt path algebras, we refer the reader to \cite{Lvt}. A quiver (directed graph) $E=(E^0,E^1,r,s)$ consists of two sets $E^0, E^1$ together with two functions $r,s:E^1\rightarrow E^0$. The elements of $E^0$ are called \textit{vertices} and those  of $E^1$ are called \textit{edges}. An edge $e\in E^1$ is called loop if $s(e)=r(e)$. A vertex $v\in E^0$ is said to be \textit{regular} if it emits a finite non-empty set of edges.\\
	A \textit{path} $\mu$ of length $n>0$ is a finite sequence of  edges $\mu=e_1e_2\dots e_n$ with $r(e_i)=s(e_{i+1})$ for all $1\leq i<n$. A vertex is regarded as a path of length $0$. We denote  by $Path(E)$ the set of all paths of $E$. \\ 
	For each edge $e\in E^1$, we consider an edge in the opposite direction, called ghost edge and denote it by $e^*$. So, we have $r(e^*)=s(e)$ and $s(e^*)=r(e)$. For each $\mu=e_1e_2\dots e_n\in Path(E)$, $\mu^*=e_n^*\dots e_2^*e_1^*$ is the corresponding ghost path.
	
	Let $E$ be an arbitrary quiver and $K$ a field. We denote by $(E^1)^*$ the set $\{e^*:e\in E^1 \}$. The \textit{Leavitt path algebra} of $E$, denoted by $L_K(E)$, with coefficients in $K$ is the free associative $K$-algebra generated by the set $E^0\cup E^1\cup (E^1)^*$, subject to the following relations:\\
	$(V)$ $vw=\delta_{v,w}v$ for all $v,w\in E^0$,\\
	$(E1)$ $s(e)e=er(e)=e$ for all $e\in E^1$,\\
	$(E2)$ $r(e)e^*=e^*s(e)=e^*$ for all $e\in E^1$,\\
	$(CK1)$ $e^*f=\delta_{e,f}r(e)$ for all $e\in E^1$, and\\
	$(CK2)$ $v=\sum_{\{e\in E^1:s(e)=v\}}ee^*$ for every regular vertex $v\in E^0$.
	
	The \textit{path $K$-algebra} of $E$, denoted by $KE$, is defined as the free associative $K$-algebra generated as an algebra by the set $E^0\cup E^1$, with relations given by $(V)$ and $(E1)$ of the definition of $L_K(E)$.
	Equivalently, the path algebra $KE$ is the algebra with basis the set of all paths in the quiver $E$ and with multiplication defined for paths $\gamma,\gamma'$ by \\
	\[
	\gamma\cdot\gamma' =
	\begin{cases}
		\gamma\gamma', & \text{if } r(\gamma)=s(\gamma'),\\
		0, & \text{otherwise}.
	\end{cases}
	\]

	\quad In this paper, $R$ represents an associative ring with unity, and all modules  are unital. We adopt the following notations: $U(R)$ for the set of all units of $R$, $Z(R)$ for the set of all central elements of $R$, $E(R)$ for the set of all idempotent elements of $R$, $J(R)$ for the Jacobson radical of $R$, $N(R)$ for the set of all nilpotent elements of $R$, $T_n(R)$ for the ring of upper triangular matrices of order $n \times n$ over $R$, $M_n(R)$ for the ring of all $n \times n$ matrices over $R$ and $\mathbb{F}_2$ for the field with $2$ elements. Moreover, we use the notation $E_{ij}$ for the matrix in $M_n(R)$ whose $(i, j)^{th}$ entry is $1$ and zero elsewhere. 
	
	\section{Zero-insertivity}
	\quad Recall that in a weakly semicommutative ring $R$, an element of the form $x=arb$ is nilpotent whenever $ab=0$, where $a,b,r\in R$. In a path algebra, every non-loop edge  can be viewed as an element of this form. For instance, let us consider the following example.
	\begin{example}\label{1ste}
		Consider the path algebra $KQ$ generated by the following quiver:
		
		\[\begin{tikzcd}
			{Q:} && \bullet v & \bullet w
			\arrow["e", from=1-3, to=1-3, loop, in=145, out=215, distance=10mm]
			\arrow["f", bend left=15, from=1-3, to=1-4]
			\arrow["h", bend left=15, from=1-4, to=1-3]
			\arrow["g", from=1-4, to=1-4, loop, in=325, out=35, distance=10mm]
		\end{tikzcd}\]
		
		Observe that each path $\gamma\in P(Q)$ can be expressed as $\gamma=s(\gamma)\gamma r(\gamma)$, with $s(\gamma)r(\gamma)=0$, whenever $s(\gamma)\neq r(\gamma)$. Moreover, there are non-monoial elements of $KQ$ which are expressible in this form, for instance, $k_1e^nf+k_2fg^m=v(k_1e^nf+k_2fg^m)w, k_1,k_2\in K$, $m,n$ are non-negative integers.
	\end{example}
	
	\begin{definition}
		We call an element $x\in R$  \textit{zero-insertive}, if $x=arb$ for some $a,b,r\in R$ with $ab=0$. We denote $Z_i(R)=\{x\in R$: $x$~{is~a~zero-insertive~element~of}~$R\}$.
	\end{definition}
	
	If $Q$ is a quiver with at least two vertices, then $\gamma=\sum_{i=1}^n k_ip_i=s(\gamma)\gamma r(\gamma)\in KQ$ is zero-insertive whenever $s(p_i)=s(p_j)\neq r(p_i)=r(p_j)$ for all $1\leq i,j\leq n$, $k_i\in K$ and $p_i\in P(Q)$. However the converse is not true, for instance, in Example \ref{1ste}, $p=efh\in Z_i(KQ)$ and $s(p)=r(p)=v$.

	By the definition of $Z_i(R)$, it follows immediate that a ring $R$ is semicommutative ( respectively, weakly semicommutative) if and only if $Z_i(R) =0~ ( \text{respectively,} ~Z_i(R)\subseteq N(R))$.
	
	




		Suppose $Q$ is a quiver consisting of a single vertex $v$ and $n$ loops $e_1,e_2,\dots, e_n$ based at $v$, that is,
		
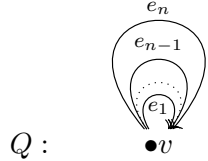
\[\begin{tikzcd}
			{Q:} & \bullet v
			\arrow["{e_1}"', from=1-2, to=1-2, loop, in=60, out=120, distance=7mm]
			\arrow[ dotted, from=1-2, to=1-2, loop, in=55, out=125, distance=10mm]
			\arrow["{e_{n-1}}", from=1-2, to=1-2, loop, in=55, out=125, distance=15mm]
			\arrow["{e_n}", from=1-2, to=1-2, loop, in=50, out=130, distance=25mm]
		\end{tikzcd}\]
		Then, $KQ\cong K\langle x_1,...x_n \rangle$, which is a domain and hence semicommutative. Thus,
		for any finite quiver $Q$, $KQ$ is semicommutative if and only if $Q$ does not possess any non-loop edges.
		\begin{proposition}\label{semZ_i(R)}
			A ring $R$ is semicommutative if and only if $ Z_i(R)\subseteq E(R)$.
		\end{proposition}
		\begin{proof}
			$(\Rightarrow)$. It is self-evident.\\ 
			$(\Leftarrow)$. Conversely, suppose that $Z_i(R)\subseteq E(R)$. Let $x,y\in R$ be such that $xy=0$. For any $r\in R$,  $yxryx\in Z_i(R)$. So by hypothesis and  $yxryx$ being a nilpotent,  $yxryx\in E(R)\cap N(R)$. Therefore,  $yxryx=0$. Since $xry\in Z_i(R)\subseteq E(R)$, $xry=(xry)^3=xr(yxryx)ry=0$. Hence, $R$ is semicommutative.	
		\end{proof}





			
			\newpage
			Throughtout this paper, let $A_n$ denote the following quiver for $n\geq 2$:
			\begin{figure}[H]
				\centering
				\[
				\begin{tikzcd}
					{A_n}: & {v_1\bullet} & {\bullet v_2} & {\bullet v_3} & {\bullet v_{n-1}} & {\bullet v_n}
					\arrow["{e_1}", bend right=-40, from=1-2, to=1-3]
					\arrow["{e_1^*}", shift left=3, dotted, bend right=-40, from=1-3, to=1-2]
					\arrow["{e_2}", bend right=-40, from=1-3, to=1-4]
					\arrow["{e_2^*}", shift left=3, dotted, bend right=-40, from=1-4, to=1-3]
					\arrow[dotted, no head, from=1-4, to=1-5]
					\arrow["{e_{n-1}}", bend right=-40, from=1-5, to=1-6]
					\arrow["{e_{n-1}^*}", shift left=3, dotted, bend right=-40, from=1-6, to=1-5]
				\end{tikzcd}
				\]
				\caption{Quiver $A_n$.}
				
				\label{fig:An}
			\end{figure}
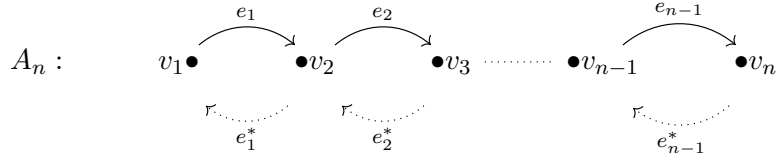
			For each edge $e_i\in A_n^1$, $e_i^*$ represents the corresponding ghost edge. \\
			A ring $R$ is said to be directly finite (\cite{Ldf}) if for any elements $a,b\in R$, $ab=1$ implies $ba=1$.
			\begin{lemma}\label{zdf}
				If a ring $R$ is directly finite, then a unit cannot be zero-insertive. 
			\end{lemma}
			\begin{proof}
				Let $u\in U(R)\cap Z_i(R)$. Then, $u=arb$, $ab=0,~ a,b,r\in R$. Since $u\in U(R)$,  $arbv=1$  for some $v\in R$. Since $R$ is directly finite, $rbva=1$, which implies $0=rbvab=b$. So, $u=0$, a contradiction.  
			\end{proof}
			
			\begin{proposition}
				Let $K$ be a field. Then, all non-unit elements of $L_K(A_2)$ are zero-insertive.
			\end{proposition}
			\begin{proof}
				
				By \cite[Lemma 1.2.12]{Lvt}, we obtain $L_K(A_2)=\{a_1v_1+a_2v_2+a_3e_1+a_4e_1^*|a_i\in K \}$. 
				Observe that $x=a_1v_1+a_2v_2+a_3e_1+a_4e_1^*\in U(L_K(E))$ with $x^{-1}=(a_1a_2-a_3a_4)^{-1}(a_2v_1+a_1v_2-a_3e_1-a_4e_1^*)$ whenever $a_1a_2\neq a_3a_4$. By Lemma \ref{zdf} and by \cite[Theorem 3.7]{Ldf},  $x\notin Z_i(L_K(A_2))$.  
				Suppose $a_1a_2=a_3a_4$. \\
				Case 1: $a_1\neq 0$. Then, $x=a_1v_1+a_2v_2+a_3e_1+a_4e_1^*=(a_1v_1+a_4e_1^*)e_1(e_1^*+a_1^{-1}a_3v_2)\in Z_i(L_K(A_n))$.\\
				Case II: $a_1=0$. Then, either $a_3=0$ or $a_4=0$. If $a_3=0$, then, \\
				$x=a_2v_2+a_4e_1^*=v_2e_1^*(a_2e_1+a_4v_1)\in Z_i(L_K(A_n))$. If $a_4=0$, then $x=a_2v_2+a_3e_1=(a_2e_1^*+a_3v_1)e_1v_2\in Z_i(L_K(A_n))$.
			\end{proof}
			\begin{lemma}\cite[Proposition 1.3.5]{Lvt} \label{L2m}
				Let $K$ be a field. Then, $M_n(K)\cong L_K(A_n)$ via the $K$-algebra isomorphism $\Phi$ defined by $\Phi(v_i)=E_{ii}, \Phi(e_i)=E_{i,i+1}, \Phi(e_i^*)=E_{i+1,i}$.
			\end{lemma}
			
			\begin{corollary}
				All non-unit elements of $M_2(K)$ are zero-insertive for any field $K$.
			\end{corollary}
			\begin{lemma}\label{Z_ip}
				Let $K$ be a field. Then, $a_1v_1+\sum_{i=2}^na_i(\prod_{j=1}^{i-1}e_{j})\in L_K(A_n)$ is zero-insertive, $a_i\in K$.
			\end{lemma} 
			\begin{proof}
				Observe that $v_1e_{n-1}^*=0$. Hence, 
				$x=a_1v_1+\sum_{i=2}^na_i(\prod_{j=1}^{i-1}e_{j})=v_1(\prod_{j=1}^{n-1}e_j)(\sum_{i=1}^{n-1}[a_i(\prod_{k=i}^{n-1}e_k)^*]+a_nr(e_{n-1}))\in Z_i(L_K(A_n))$. 
			\end{proof}

			\newpage
			In \cite{Diesl}, Diesl introduced the concept of nil clean rings (also see \cite{ncmr}, \cite{Kosan}). An element $x$ of a ring  $R$ is said to be nil-clean  if $x=e+a$, where $e\in E(R), a\in N(R)$. In any path algebra $KQ$ corresponding to the quiver $Q$, nil-clean elements occur naturally in abundance. Indeed, for any vertex $v\in Q^0$  and non-loop edge $f\in Q^1,$ $v+f\in KQ$ is nil-clean. In the Leavitt path algebra
			$L_{\mathbb{F}_2}(A_n)$, we observe that each $v_i$ is  zero-insertive, as,
			\[
			\begin{cases}
				v_j=v_j e_j e_j^{*},\quad v_j e_j^{*}=0,
				& 1\leq j\leq n-1,\\[2mm]
				v_n=v_n e_{n-1}^{*}e_{n-1},\quad v_n e_{n-1}=0.
			\end{cases}
			\]
			In addition, for $2\leq i\leq n$,
			\[
			v_i=e_{i-1}^{*}+\bigl(e_{i-1}^{*}+v_i\bigr),
			\]
			while
			\[
			v_1=e_1+\bigl(e_1+v_1\bigr),
			\]
			where, in each case, the first summand is nilpotent and the second is
			idempotent, that is, $v_i$ is nil-clean for each $i$. Hence, each zero-insertive element $v_i$ is  nil-clean. This naturally raises the following question: is every zero-insertive element of $L_{\mathbb{F}_2}(A_n)$ essentially nil-clean ? The following theorem answers this question.		
			\begin{theorem}\label{dsnc}
				Let $K$ be a field and $n$ be an integer $\geq 2$. Then, each zero-insertive element of $L_K(A_n)$ is nil-clean if and only if $K\cong \mathbb{F}_2$. 
			\end{theorem} 
			
			\begin{proof}
				Suppose each element of $Z_i(L_K(A_n))$ is a nil-clean. Assume, if possible, that $K\not\cong\mathbb{F}_2$. Then, there exists $a \in K\setminus \{0,1\}$.
				By Lemma \ref{Z_ip},   $av_1\in Z_i(L_K(A_n))$. So, $av_1$ is nil-clean. Observe that by lemma \ref{L2m}, $av_1$ correspond to $aE_{11}\in M_n(K)$. By \cite[Theorem 3]{everymatrix}, our assumption leads to a contradiction. Hence, $K\cong \mathbb{F}_2$. On the other hand, the converse follows from \cite[Theorem 3]{ncmr} and \cite[Proposition 1.3.5]{Lvt}. 
			\end{proof}
			\begin{definition}
				We call a ring $R$ \textit{zero-insertive nil-clean} ($ZINC$) if its zero-insertive elements are nil-clean.   
			\end{definition}
			\begin{corollary}
				$L_{\mathbb{F}_2}(A_n)$ is a ZINC ring for all positive integer $n\geq 2$.
			\end{corollary}
			
			\begin{lemma}\label{wsn}
				Let $R$ be a  ZINC ring. If $R$ has no non-trivial idempotent elements, then $Z_i(R)\subseteq N(R)$.
			\end{lemma}
			\begin{proof}
				Let $x\in Z_i(R)$. Then, $x=arb$ for some $a,b,r\in R$ and $ab=0$. Since $R$ is ZINC, either  $x$ is nilpotent or $x=1+h$ for some $h\in N(R)$. It is well known that an Abelian ring (that is, a ring in which all idempotents are central) is directly finite. Hence, if $x=1+h$, then $x\in U(R)$, which is a contradiction by Lemma \ref{zdf}. 
				
			\end{proof}

			However, the converse is not true, for example, $\mathbb{Z}_6$ is ZINC and $Z_i(\mathbb{Z}_6)=0\subseteq N(\mathbb{Z}_6)$ but $E(\mathbb{Z}_6)=\{0,1,3,4\}.$

			\begin{proposition}
				Suppose $R$ is a ring with no non-trivial idempotent. Then, for $n\geq 1$, $T_n(R)$ is ZINC if and only if $Z_i(T_n(R))\subseteq N(T_n(R))$.
			\end{proposition}
			\begin{proof}
				Suppose $T_n(R)$ is ZINC, that is, $Z_i(T_n(R))\subseteq E(T_n(R))+N(T_n(R))$.	Let $P=\begin{pmatrix}
					p_{11} &  p_{12}&  \dots &  p_{1n}\\
					0 &  p_{22} & \dots &  p_{2n}\\
					\vdots & \vdots& \ddots & \vdots \\
					0 & 0 & \dots &  p_{nn}
				\end{pmatrix}\vspace{0.2cm}\in Z_i(T_n(R))$. This implies that each $p_{ii}\in Z_i(R)$, $1\leq i\leq n$.  As $Z_i(T_n(R))\subseteq E(T_n(R))+N(T_n(R))$, there exists $F=(f_{ij})\in E(T_n(R))$ such that $P-F\in N(T_n(R))$, that is, $p_{ii}-f_{ii}\in N(R)$, $1\leq i\leq n$. By hypothesis, $f_{ii}=0$ or $1$. By the proof of Lemma \ref{wsn}, we get that $p_{ii}\in N(R)$, that is, $P\in N(T_n(R))$. So, $Z_i(T_n(R))\subseteq N(T_n(R))$. Whereas the converse is self-evident.             
			\end{proof}

			According to (\cite{sjc}), $a\in R$ is said to be strongly J-clean  if it can be expressed as the sum of an idempotent element and an element in its Jacobson radical that commute. A ring $R$ is said to be J-clean if for any $x\in R$, $x=e+j$ for some $e\in E(R), j\in J(R)$.
			
			\begin{theorem}
				Let $R$ be a J-clean ring. If $J(R)$ is nil, then  $Z_i(R)\subseteq N(R)$.
			\end{theorem}
			
			\begin{proof}
				Let $a\in Z_i(R)$. So,  $a=xry$  for some $x,y,r\in R$ with $xy=0$. By hypothesis, there exists $e\in E(R)$ such that $a-e\in J(R)$. If $e=0$, we are done. Assume, if possible, that $e\neq 0$. Since $yx\in N(R)$, $1-yx\in U(R)$. As $R$ is J-clean, $1-yx-e_1\in J(R)$ for some $e_1\in E(R)$. Therefore, $1-(1-yx)^{-1}e_1\in J(R)$. This yields that $(1-yx)^{-1}e_1\in U(R)$ and hence $e_1\in U(R)$, that is, $e_1=1$. This implies that $yx\in J(R)$. Hence, $a^2\in J(R)$. Observe that
				$(a-e)^2=a^2-ae-e(a-e)$.  So, $ae\in  J(R)$. Note that $a-e=j$ for some $j\in J(R)$. Hence $e=ae-je\in J(R)$, that is, $e=0$, a contradiction. Thus, $a\in J(R)$, that is, $Z_i(R)\subseteq N(R)$.
			\end{proof}
			
			\begin{proposition}\label{sj} 
				Let $R$ be a  ZINC ring.  If $ab=0$, then $aJ(R)b\subseteq N(R)$, $a,b\in R$.
				
			\end{proposition}
			\begin{proof}
				Let $a,b\in R$ be such that $ab=0$. Since $R$ is ZINC, $aJ(R)b\subseteq Z_i(R) \subseteq E(R)+N(R)$. So, for any $j\in J(R)$,  there exist $e\in E(R)$, $h\in N(R)$ such that $ajb=e+h$.  Then $(ajb-e)^{m}=0$ for some positive integer $m$. This implies that $j_1-e=0$ for some $j_1\in J(R)$.  So, $e\in J(R)$, that is, $e=0$.
			\end{proof}
			
			Now, we illustrate an example of a ZINC ring whose polynomial extension and power series extension fail to be  ZINC.
			\begin{example}
				
				\begin{enumerate}
					
					\item Take $R=M_2(\mathbb{Z}_2)$, which is ZINC by Lemma \ref{L2m} and  Theorem \ref{dsnc}.  Observe that $f(x)=E_{12}$, $g(x)=E_{11}x\in R[x]$ and $f(x)g(x)=0$. So, $f(x)E_{21}g(x)=g(x)\in Z_i(R[x])$. Suppose $g(x)=e(x)+n(x)$, where $e(x)=e_0+e_1x+\dots +e_kx^k\in E(R[x]), n(x)=n_0+n_1x+\dots +n_{k_1}x^{k_1}\in N(R[x])$. Clearly, $e_0=-n_0\in E(R)\cap N(R)$, that is, $e_0=0$. If $e(x)\neq 0$, then $e(x)=x^mh(x)$, where $m\geq 1$ and $h(x)=h_0+h_1x+h_2x^2+\dots +h_{k_2}x^{k_2} \in R[x]$ and $h_0\neq 0$. Comparing the coefficient of $x^m$ in $e(x)^2=e(x)$ gives $h_0=0$, a contradiction. Thus $g(x)\in N(R[X])$, a contradiction. Therefore, $R[x]$ is not ZINC.
					
					\item  Take $R=M_2(\mathbb{Z}_2)$ which is ZINC (by Theorem \ref{dsnc}).  It is well known that $J(R[[x]])=J(R)+\langle x \rangle$. Observe that $E_{11}E_{21}=0$ and $E_{11}(E_{12}x)E_{21}=E_{11}x\notin N(R[[x]])$. By Proposition \ref{sj}, $R[[x]]$ is not ZINC.  
				\end{enumerate}
				
			\end{example}

			\begin{theorem}\label{nil}
				Let $R$ be a ring and $I$ be a nil ideal of $R$. If $R/I$ is  ZINC, then $R$ is  ZINC.
			\end{theorem}
			
			\begin{proof}
				Suppose that $R/I$ is ZINC. Let $x\in Z_i(R)$. Clearly, $\bar{x}\in Z_i(R/I)$.  Since $R/I$ is  ZINC,  $\bar{x}=\bar{e}+\bar{h}$ for some $\bar{e}\in E(R/I), \bar{h}\in N(R/I)$. Since idempotents lifts modulo any nil ideal, $\bar{e}=\bar{f}$ for some $f\in E(R)$. Then $x-f$ is nilpotent modulo $I$. Thus, $(x-f)^m\in I\subseteq N(R)$ for some positive integer $m$. So, $x-f\in N(R)$. Therefore, $R$ is  ZINC.
			\end{proof}
			
			\begin{proposition}\label{prod}
				Let $\{R_{\alpha} \}$ be a finite collection of rings. Then the direct product $R=\prod R_{\alpha}$  is  ZINC if and only if each $R_{\alpha}$ is  ZINC.
			\end{proposition}
			
			\begin{proof}
				The proof is trivial.	
			\end{proof}
			
			However, it is possible to  construct an infinite direct product of ZINC rings that is not  ZINC, as illustrated in the following Example \ref{infinite pdt}. The following remark will be referred in Example \ref{infinite pdt} and other subsequent results.
			
			\begin{remark}\label{ZINC rem}
				For any $w\in R$, observe that $wE_{11}\in Z_i(M_n(R))$ as $wE_{11}=E_{1n}(wE_{n1})E_{11}$ for all $n\geq 2$.
				
			\end{remark} 
			
			\begin{example}\label{infinite pdt}
				It is easy to observe that $\mathbb{Z}_{2^m}$ is nil clean for all positive integer $m$. By \cite[Corollary 7]{ncmr}, $M_n(\mathbb{Z}_{2^{m}})$ is nil clean for any positive integers $n$ and $m$. Now, take $R=M_2(\mathbb{Z}_{2})\times M_3(\mathbb{Z}_{2^{2}})\times M_4(\mathbb{Z}_{2^{3}})\times \cdots=\prod_{i=1}^{\infty} M_{i+1}(\mathbb{Z}_{2^{i}})$. Let $P=(2E_{11}, 2E_{11},\dots)$. By Remark \ref{ZINC rem}, $P\in Z_i(R)$. Note that $J(\prod_{i=1}^{\infty} M_{i+_1}(\mathbb{Z}_{2^{i}}))=\prod_{i=1}^{\infty} M_{i+_1}(J(\mathbb{Z}_{2^{i}}))$. Let $K_i=2E_{11}\in M_{i+_1}(\mathbb{Z}_{2^{i}})$. Then $K_i\in J( M_{i+_1}(\mathbb{Z}_{2^{i}}))$ as $2\in J(\mathbb{Z}_{2^i})$. Suppose $K_i=F+B$, where $F$ is an idempotent and $B$ is a nilpotent element. Let $V=I+B\in U(M_{i+1}(\mathbb{Z}_{2^i}))$. So, $V=(I-F)+K_i$. This implies that $V=V(I-F)V^{-1}+VKV^{-1}$. So, we obtain $V(I-K_iV^{-1})=V(I-F)V^{-1}\in E(M_{i+1}(\mathbb{Z}_{2^i}))\cap U(M_{i+1}(\mathbb{Z}_{2^i}))$. Therefore, $F=0$. Note that $K_i$ is nilpotent of index $i$. Hence, $P\notin E(R)+N(R)$  as $P=(K_1, K_2,\dots)$. Thus, $R$ is not a  ZINC ring.   
			\end{example}

			\begin{definition}
				Let $M$ be an $(R,R)$-bimodule. The \textit{trivial extension} of $R$ by $M$ is the ring $R\propto M=R\oplus M$, where the addition is usual, and the multiplication is defined as 
				$(r_1,s_1)(r_2,s_2)=(r_1r_2,r_1s_2+s_1r_2),r_i\in R,s_i\in M$ and $i=1,2$.     
				
			\end{definition} 
			
			The following corollary is a  consequence of Theorem \ref{nil} and Proposition \ref{prod}.

			\begin{corollary} 
				\begin{enumerate}
					\item  $T_n(R)$ is  ZINC if and only if $R$ is  ZINC.
					\item Let $M$ be an $(R,R)$-bimodule. The trivial extension $R\propto M$ is ZINC if and only if $R$ is ZINC.
					\item $R[x]/\langle x^n \rangle$ is ZINC if and only if $R$ is ZINC.

				\end{enumerate}
			\end{corollary}
			\begin{proof}
				For  $(1)$: Suppose $R$ is  ZINC. Observe that $I=\begin{pmatrix}
					0 & R & \dots & R\\
					0 & 0 & \dots & R \\
					\vdots & \vdots & \ddots & \vdots \\
					0 & 0 & \dots & 0
				\end{pmatrix}$ is a nil ideal of $ T_n(R)$.  Note that $T_n(R)/I\cong R\times \dots \times R$.
				By Theorem \ref{nil} and Proposition \ref{prod}, $T_n(R)$ is ZINC. Conversely, suppose $T_n(R)$ is ZINC.  If $x\in Z_i(R)$, then $A=\begin{pmatrix}
					x & \mathbf{0}\\
					\mathbf{0} & \mathbf{0}
				\end{pmatrix}\in Z_i(T_n(R))$. Since $T_n(R)$ is ZINC, $x\in E(R)+N(R)$.
				The proofs of $(2)$ and $(3)$ follows analogously as $R\propto M\cong \left\{\left(\begin{array}{rr}
					t & s\\
					0 & t
				\end{array}
				\right) : t\in R, s\in M \right\}$ and  $R[x]/<x^n>\cong  
				\left \lbrace
				\left(\begin{array}{lccccr}
					r_1 & r_{2} & r_3 & \dots & r_{n-1} & r_n\\
					0 & r_1 & r_2 & \dots &  r_{n-2} & r_{n-1}\\
					\vdots & \vdots  & \vdots & \vdots & \vdots \\
					0 & 0 & 0 & \dots & r_1 & r_2 \\
					0 & 0 & 0 & \dots & 0 & r_1
				\end{array}
				\right ):r_i\in R\right \rbrace$ .
			\end{proof}

			\begin{corollary}
				Let $K$ be a field. Then, the path algebra $KA_n$ is ZINC (\cite[Example 1.13]{rept}). 
			\end{corollary}

			\begin{lemma}\label{nztc}
				If $R$ is ZINC, then $Z_i(R)\subseteq E(R)+U(R)$.
			\end{lemma}
			\begin{proof}
				Let $x\in Z_i(R)$. Clearly, $-x\in Z_i(R)$. As $R$ is ZINC, $-x=e+h$ for some $e\in E(R)$  and $h\in N(R)$. Therefore, $x=(1-e)+(-h-1)\in E(R)+U(R)$.	
			\end{proof}
			
			We skip the proof of the following trivial Lemma. 
			\begin{lemma}\label{invertible}
				Let $R$ be a ring. Then,   $U=\begin{pmatrix}
					1 & u_{12} & \dots & u_{1n}\\
					0 & 1 & \dots & u_{2n} \\
					\vdots & \vdots & \ddots & \vdots \\
					0 & 0 & \dots & 1
				\end{pmatrix}, V=\begin{pmatrix}
					1 & 0 & \dots & 0\\
					v_{21} & 1 & \dots & 0 \\
					\vdots & \vdots & \ddots & \vdots \\
					v_{n1} & v_{n2} & \dots & 1
				\end{pmatrix}\in M_n(R)$ are invertible.
			\end{lemma}

			\begin{lemma}\label{ZINCmatrix}
				Suppose $X=\begin{pmatrix}
					1 & u_{12} & \dots & u_{1n}\\
					0 & 1 & \dots & u_{2n} \\
					\vdots & \vdots & \ddots & \vdots \\
					0 & 0 & \dots & 1
				\end{pmatrix}$, $Y=\begin{pmatrix}
					1 & 0 & \dots & 0\\
					-x_2 & 1 & \dots & 0 \\
					\vdots & \vdots & \ddots & \vdots \\
					-x_n & 0 & \dots & 1
				\end{pmatrix}$ and $U=(u_{ij})\in M_n(R)$. If $XYU=(b_{ij})$, then $b_{1j}=(1-\sum\limits_{i=2}^nu_{1i}x_i)u_{1j}+ \sum\limits_{k=2}^nu_{1k}u_{kj}$, when $j\geq 1$ and $	b_{ij}=(-x_i-\sum\limits_{k=i+1}^nu_{ik}x_k)u_{1j}+u_{ij}+\sum\limits_{k=i+1}^nu_{ik}u_{kj}$, when $i> 1$, $j\geq 1$.
			\end{lemma}
			
			A ring $R$ is said to be weakly clean (\cite{wclean}) if for any $x\in R$ there exist $e\in E(R), u\in U(R)$ such that $x-e-u\in (1-e)Rx$.
			
			\begin{theorem}\label{weaklyc}
				For any $n\geq 2$, if $M_n(R)$ is a ZINC ring, then $R$ is weakly clean ring.
			\end{theorem}
			\begin{proof}
				Let $w\in R$ and $ W=wE_{11}\in M_n(R)$. By Remark \ref{ZINC rem}, $W\in Z_i(M_n(R))$. Since  $M_n(R)$ is ZINC,   $W=F+V$ for some $F\in E(M_n(R)$ and $V\in U(M_n(R))$. This implies that $V^{-1}W=V^{-1}F+I$, where $I$ is the identity matrix. So, we obtain \vspace{0.2cm} $V^{-1}W=(V^{-1}FV)V^{-1}=GV^{-1}+I$, where  $G=V^{-1}FV$ which is also idempotent. Hence, $(I-G)V^{-1}W=I-G$. As $W=wE_{11}$, $I-G=\begin{pmatrix}
					g & 0 & \dots & 0\\
					x_2 & 0 & \dots & 0 \\
					\vdots & \vdots & \ddots & \vdots \\
					x_n & 0 & \dots & 0
				\end{pmatrix}\vspace{0.2cm}$ for some $x_i,g\in R$, $2\leq i\leq n$.  Since $I-G$ is idempotent,   $g\in E(R), x_i\in Rg$, and  $i>1$. So,
				
				\begin{equation}\label{ZINCG}
					G=\begin{pmatrix}
						f & 0 & \dots & 0 & 0\\
						-x_2 & 1 & \dots & 0 & 0 \\
						\vdots & \vdots & \ddots & \vdots & \vdots \\
						-x_{n-1} & 0 & \dots & 1 & 0\\
						-x_n & 0 & \dots & 0 &  1
					\end{pmatrix}
				\end{equation}
				where $f=1-g$.
				Let $V^{-1}=(u_{ij})$. Since $V^{-1}W=GV^{-1}+I$ \vspace{0.2cm}, \\ $\begin{pmatrix}
					u_{11}w & 0 & \dots & 0\\
					u_{21}w & 0 & \dots & 0 \\
					\vdots & \vdots & \ddots & \vdots \\
					u_{n1}w & 0 & \dots & 0
				\end{pmatrix}=\begin{pmatrix}
					fu_{11}+1 & fu_{12} & \dots & fu_{1n}\\
					-x_2u_{11}+u_{21} & -x_2u_{12}+u_{22}+1 & \dots & -x_2u_{1n}+u_{2n} \\
					\vdots & \vdots & \ddots & \vdots \\
					-x_nu_{11}+u_{n1} & -x_nu_{12}+u_{n2} & \dots & -x_nu_{1n}+u_{nn}+1
				\end{pmatrix}\vspace{0.2cm}$ (see equation \ref{ZINCG}). So, we obtain the following equalities \ref{ZINC1}, \ref{ZINC2}, \ref{ZINC3}, \ref{ZINC4} and \ref{ZINC5} :
				
				\begin{equation}\label{ZINC1}
					fu_{11}+1=u_{11}w
				\end{equation} For $j>1$, we have
				\begin{equation}\label{ZINC2}
					fu_{1j}=0=(1-g)u_{1j}	
				\end{equation} For $1<j\leq n$, we have 
				\begin{equation}\label{ZINC3}
					-x_ju_{1j}+u_{jj}=-1	
				\end{equation} For $1<i\leq n$, we have
				\begin{equation}\label{ZINC4}
					-x_iu_{11}+u_{i1}=u_{i1}w	
				\end{equation}
				For all $i\neq j$,  $1<i,j\leq n$ 
				
				\begin{equation}\label{ZINC5}
					-x_iu_{1j}+u_{ij}=0	
				\end{equation}
				
				If $B=\begin{pmatrix}
					1 & u_{12} & \dots & u_{1n}\\
					0 & 1 & \dots & u_{2n} \\
					\vdots & \vdots & \ddots & \vdots \\
					0 & 0 & \dots & 1
				\end{pmatrix}\begin{pmatrix}
					1 & 0 & \dots & 0\\
					-x_2 & 1 & \dots & 0 \\
					\vdots & \vdots & \ddots & \vdots \\
					-x_n & 0 & \dots & 1
				\end{pmatrix}V^{-1}$, then   $B$  is invertible (see Lemma \ref{invertible}). Let $B=(b_{ij})$. We claim that $b_{11}\in U(R)$. Now, for determining $b_{ij}$, we use the Lemma \ref{ZINCmatrix} freely. So, we get

				$\begin{aligned}
					b_{11} & =(1-\sum\limits_{i=2}^nu_{1i}x_i)u_{11}+ \sum\limits_{k=2}^nu_{1k}u_{k1}\\ 
					& =u_{11}+u_{12}(-x_2u_{11}+u_{21})+u_{13}(-x_3u_{11}+u_{31})+\dots +u_{1n}(-x_nu_{11}+u_{n1})
				\end{aligned}$
				
				Therefore, by equation \ref{ZINC4}
				\begin{equation}\label{ZINC6}
					b_{11}=u_{11}+\sum\limits_{i=2}^nu_{1i}u_{i1}w	
				\end{equation} For $j> 1$, we have
				
				$\begin{aligned}
					b_{1j} & =(1-\sum\limits_{i=2}^nu_{1i}x_i)u_{1j}+ \sum\limits_{k=2}^nu_{1k}u_{kj}\\
					&=u_{1j}+u_{12}(-x_2u_{1j}+u_{2j})+\dots +u_{1j}(-x_ju_{1j}+u_{jj})+\dots +u_{1n}(-x_nu_{1j}+u_{nj})	
				\end{aligned}$
				Therefore, by equation  \ref{ZINC3}, equation \ref{ZINC5}, for $j>1$
				
				\begin{equation}\label{ZINC7}
					b_{1j}=0
				\end{equation}

				For $1<j<i$, 
				
				$\begin{aligned}
					b_{ij}& =(-x_i-\sum\limits_{k=i+1}^nu_{ik}x_k)u_{1j}+u_{ij}+\sum\limits_{k=i+1}^nu_{ik}u_{kj}\\
					&=-x_iu_{1j}+u_{ij}+\sum\limits_{k=i+1}^nu_{ik}(-x_ku_{1j}+u_{kj}).	
				\end{aligned}$ \\
				
				So, by equation \ref{ZINC5},  for $1<j<i$
				\begin{equation}
					b_{ij}=0
				\end{equation}
				
				For $i\neq 1$, we have
				
				$\begin{aligned}
					b_{ii} &=(-x_i-\sum\limits_{k=i+1}^nu_{ik}x_k)u_{1i}+u_{ii}+\sum\limits_{k=i+1}^nu_{ik}u_{ki}\\
					& =-x_iu_{1i}+u_{ii}+\sum\limits_{k=i+1}^nu_{ik}(-x_ku_{1i}+u_{ki}).
				\end{aligned}$

				Therefore, by equation \ref{ZINC3} and equation \ref{ZINC5}, for $i\neq 1$
				\begin{equation}
					b_{ii}=-1
				\end{equation}
				
				For $1<i<j$,
				
				$\begin{aligned}
					b_{ij} &=(-x_i-\sum\limits_{k=i+1}^nu_{ik}x_k)u_{1j}+u_{ij}+\sum\limits_{k=i+1}^nu_{ik}u_{kj}\\
					&=-x_iu_{1j}+u_{ij}+\sum\limits_{{k=i+1},k\neq j}^nu_{ik}(-x_ku_{1j}+u_{kj})+u_{ij}(-x_ju_{ij}+u_{jj}).
				\end{aligned}$

				Hence, by equation \ref{ZINC3} and equation \ref{ZINC5}, $1<i<j$, we have 
				\begin{equation}
					b_{ij}=-u_{ij}
				\end{equation}

				Therefore, $B=\begin{pmatrix}
					u_{11}+\sum\limits_{i=2}^nu_{1i}u_{i1}w & 0 & 0 & \dots & 0 & 0\\
					b_{21} & -1 & -u_{23} & \dots & -u_{2,n-1} & -u_{2n} \\
					\vdots & \vdots & \vdots & \vdots & \vdots & \vdots \\
					b_{n-1,1} & 0 & 0 & \dots & -1 & -u_{n-1,n}\\
					b_{n1} & 0 & 0 & \dots & 0 & -1
				\end{pmatrix}\vspace{0.2cm}$. 
				
				Take $c_{i1}=b_{i1}+\sum\limits_{j=i+1}^nb_{ij}(b_{j1}+\sum\limits_{k=j+1}b_{jk}b_{k1})$ and  $C=\begin{pmatrix}
					1 & 0 & 0 & \dots & 0 & 0\\
					c_{21} & 1 & 0 & \dots & 0 & 0 \\
					\vdots & \vdots & \vdots & \vdots & \vdots & \vdots \\
					c_{n-1,1} & 0 & 0 & \dots & 1 & 0\\
					c_{n1} & 0 & 0 & \dots & 0 & 1
				\end{pmatrix}\vspace{0.3cm}$. Note that $B$ and $C$ are units (see Lemma \ref{ZINCmatrix}) and so is\\
				$BC=\begin{pmatrix}
					b_{11} & 0 & 0 & \dots & 0  & 0\\
					0 & -1 & -u_{23} & \dots  & -u_{2,n-1} & -u_{2n} \\
					\vdots & \vdots & \vdots & \vdots & \vdots & \vdots \\
					0 & 0 & 0 & \dots & -1 & -u_{n-1,n}\\
					0 & 0 & 0 & \dots & 0 & -1
				\end{pmatrix}\vspace{0.2cm}\in U(M_n(R))$. Hence, $b_{11}=u_{11}+\sum\limits_{i=2}^nu_{1i}u_{i1}w\in U(R)$. Substituting $u_{11}=b_{11}-(\sum\limits_{i=2}^nu_{1i}u_{i1})w$ in equation \ref{ZINC1}, we get $f(b_{11}-(\sum\limits_{i=2}^nu_{1i}u_{i1})w)+1=(b_{11}-(\sum\limits_{i=2}^nu_{1i}u_{i1})w)w$. By equation \ref{ZINC2}, we obtain that $fb_{11}+1=b_{11}w-(\sum\limits_{i=2}^nu_{1i}u_{i1})w^2$, that is, $w-b_{11}^{-1}fb_{11}-b_{11}^{-1}=b_{11}^{-1}(\sum\limits_{i=2}^nu_{1i}u_{i1})w^2=b_{11}^{-1}(\sum\limits_{i=2}^ngu_{1i}u_{i1})w^2$ (by equation \ref{ZINC2}, $gu_{1i}=u_{1i}$ for all $i>1$). So, $w-e-v=(1-e)rw^2$, where $e=b_{11}^{-1}fb_{11}$, $v=b_{11}^{-1}$ and $r=(\sum\limits_{i=2}^nvu_{1i}u_{i1})$. Therefore, $R$ is weakly clean. 
			\end{proof}
			
			\begin{corollary}\label{pm}
				If $R$ is a domain such that $M_n(R)$ is ZINC, then for each $w\in R$ there exists $v\in U(R)$ such that $w-v=rw^2$ or $w-v=1$ for some $r\in R$.
			\end{corollary}

			\begin{corollary}
				For any integer $n>1$, $M_n(\mathbb{Z})$ is not ZINC.
			\end{corollary}
			\begin{proof}
				For any $u\in U(\mathbb{Z})=\{\pm 1\}$, we have 
				$3-u=2$ or $4$. By Corollary \ref{pm}, $M_n(\mathbb{Z})$ is not ZINC. 
			\end{proof}
			\begin{remark}
				Observe that all commutative rings are ZINC but need not be weakly clean, for instance, $\mathbb{Z}$. Therefore, Theorem \ref{weaklyc} is not true for $n=1$, in general.
			\end{remark} 
			
			\begin{remark}
				Any division ring is weakly clean. If $R$ is a division ring with $|R|>2$  then $M_n(R)$ is not ZINC (by Theorem \ref{dsnc}), where  $|R|$ denotes  the cardinality of $R$. Therefore, the converse of Theorem \ref{weaklyc} is not true.
			\end{remark}

			A \textit{Morita context} (\cite{morita}) is a $4$-tuple $\begin{pmatrix}
				R_1 & M\\
				P & R_2
			\end{pmatrix}$, where $R_1$, $R_2$ are rings, $M$ is $(R_1,R_2)$-bimodule and $P$ is $(R_2,R_1)$-bimodule, and there exists a context product $M\times P\rightarrow R_1$ and $P\times M\rightarrow R_2$ written multiplicatively as $(m,p)\mapsto mp$ and $(p,m)\mapsto pm$. Clearly, $\begin{pmatrix}
				R_1 & M\\
				P & R_2
			\end{pmatrix}$ is an associative ring with the usual matrix operations.\\
			A Morita context $\begin{pmatrix}
				R_1 & M\\
				P & R_2
			\end{pmatrix}$ is said to be trivial if the context products are trivial, that is, $MP=0$ and $PM=0$.

			\begin{theorem}
				Let $R=\begin{pmatrix}
					A & M\\
					N & B
				\end{pmatrix}$ be a trivial Morita context. If $A$ and $B$ are ZINC, then $R$ is ZINC.
				
			\end{theorem}
			\begin{proof}
				Suppose $A$ and $B$ are  ZINC rings. Let $S\in Z_i(R)$. Then, $S=PXQ$, where  $P=\begin{pmatrix}
					a_0 & m_0 \\
					n_0 & b_0
				\end{pmatrix} $, $Q=\begin{pmatrix}
					a_1 & m_1 \\
					n_1 & b_1
				\end{pmatrix}, X=\begin{pmatrix}
					x & m_2 \\
					n_2 & y
				\end{pmatrix} \in R$ with $PQ=0.$ Since $PQ=0$, $a_0a_1=0$ and $b_0b_1=0$. Observe that $a_0xa_1\in Z_i(A)$ and $b_0yb_1\in Z_i(B)$. Since $A$ and $B$ are ZINC rings, there exist $e_x\in E(A),~e_y\in E(B)$, $n_x\in N(A),~m_y\in N(B)$ such that $a_0xa_1=e_x+n_x$\vspace{0.2cm} and $b_0yb_1=e_y+n_y$.
				Now, $PXQ=\begin{pmatrix}
					e_x & 0 \\
					0 & e_y
				\end{pmatrix}\vspace{0.2cm} +\begin{pmatrix}
					n_x & m \\
					n & n_y
				\end{pmatrix}$ for some $m\in M$ and $n\in N$. Observe that $\begin{pmatrix}
					e_x & 0 \\
					0 & e_y
				\end{pmatrix}\in E(R)$, $\begin{pmatrix}
					n_x & m \\
					n & n_y
				\end{pmatrix}\in N(R)$. Therefore, $R$ is ZINC.
				
			\end{proof}

			\begin{proposition}
				The following are equivalent:
				\begin{enumerate}
					\item $R$ is ZINC. 
					\item $Re$ is ZINC for all central idempetent $e$.
					\item There exists a central idempotent $e$ of $R$ such that both $Re$ and $R(1-e)$ are ZINC.
				\end{enumerate}
			\end{proposition}
			\begin{proof}
				It is easy to observe that $(1)\Rightarrow (2)\Rightarrow (3)$.
				Now, for $(3)\Rightarrow (1)$, let $x\in Z_i(R)$. Observe that $x=xe+x(1-e)$. As $e$ is central, $xe\in Z_i(Re)$ and $x(1-e)\in Z_i(R(1-e))$. By hypothesis, $x=(f+p)+(g+q)$, where $f\in E(Re), p\in N(Re), g\in E(R(1-e)), q\in N(R(1-e))$. Hence, $x=(f+g)+(p+q)\in E(R) +N(R)$. 
			\end{proof}

		\end{document}